\documentclass[11pt]{article}
\usepackage{mathrsfs} 
\usepackage[margin=1.25in]{geometry}
\usepackage{ifthen}
\usepackage{graphicx}
\usepackage{epsfig}
\usepackage{dsfont, amssymb}
\usepackage{amsfonts,dsfont,amssymb,amsthm,stmaryrd,bbm}
\usepackage[cmex10]{amsmath} % Use the [cmex10] option 
\usepackage[hidelinks]{hyperref,mathtools}

\newtheorem{lemma}{Lemma}%[section]

\newtheorem{conj}{Conjecture}[section]
\newtheorem{theorem}{Theorem}[section]

\newtheorem{remark}[conj]{Remark}

\newtheorem{definition}[conj]{Definition}
\newtheorem{cor}[conj]{Corollary}

\newcommand\independent{\protect\mathpalette{\protect\independent}{\perp}} 
\def\independent#1#2{\mathrel{\rlap{$#1#2$}\mkern2mu{#1#2}}}

\newcommand{\R}{\mathbb{R}}

\newcommand{\A}{\alpha}

\newcommand{\Z}{\mathbb{Z}}

\def\phi{\varphi}

\def\bee{\begin{eqnarray*}}
\def\ene{\end{eqnarray*}}

\usepackage{color}

\title{Reversals of R\'enyi Entropy Inequalities under Log-Concavity}

\author{
  James Melbourne\\
    University of Minnesota\\
  \texttt{melbo013@umn.edu}
  \and
  Tomasz Tkocz\\
  Carnagie Mellon University\\
  \texttt{ttkocz@andrew.cmu.edu}
}

\begin{document}

\maketitle

\begin{abstract}
We establish a discrete analog of the R\'enyi entropy comparison due to Bobkov and Madiman.  For log-concave variables on the integers, the min entropy is within $\log e$ of the usual Shannon entropy.  Additionally we investigate the entropic Rogers-Shephard inequality studied by Madiman and Kontoyannis, and establish a sharp R\'enyi version for certain parameters in both the continuous and discrete cases. 
\end{abstract}

\section{Introduction}
The R\'enyi entropy \cite{Ren61} is a family of entropies parameterized by an $\A$ belonging to the extended non-negative real line that generalizes several important notions of uncertainty, most significantly the Shannon entropy when $\alpha=1$, but also the min entropy $\alpha=\infty$, the collision entropy $\alpha= 2$, and the Hartley or max-entropy $\alpha= 0$.  
\begin{definition}
    For a random variable $X$ taking countably many values with probabilities $p_i >0$, and $\alpha \in (0,1) \cup (1,\infty)$, the R\'enyi entropy $H_\alpha (X)$ is
\begin{equation}
H_\alpha(X)= \frac{ \log \sum_i p_i^\alpha}{1- \A},
\end{equation}
while the remaining $\A$ are defined through continuous limits.  $H_0(X) = \log|\{i : p_i > 0\}|$, $H_1(X) = H(X) = - \sum_i p_i \log p_i$, and $H_\infty(X) = - \log \| p\|_\infty$, where $| \cdot |$ denotes cardinality, and $\| p \|_\infty$ denotes the $\ell_\infty$ norm of the sequence $p_i$.
\end{definition}

Similar definitions are put forth in the continuous setting.

\begin{definition}
    For an $\mathbb{R}^d$-valued random variable $Y$ with density function $f$, and $\A \in (0,1) \cup (1,\infty)$, the R\'enyi entropy $h_\alpha(Y)$ is
\begin{equation}
h_\alpha(Y)= \frac{ \log \int_{\mathbb{R}^d} f^\alpha(x) dx}{1- \alpha},
\end{equation}
while the remaining $\alpha$ are defined through continuous limits.  $h_0(X) = \log|\{x : f(x)>0\}|$, $h_1(X) = h(X) = - \int f(x) \log f(x) dx$, and $h_\infty(X) = - \log \| f\|_\infty$, where $| \cdot |$ denotes Lebesgue volume, and $\| f \|_\infty$ denotes the $L_\infty$ norm with respect to the Lebesgue measure.
\end{definition}

The R\'enyi entropy can usefully be described in terms of the $\alpha-1$-norm\footnote{This is not truly a norm when $\alpha < 2$},
\begin{align}
    H_\alpha(X) = \frac{\log \sum_i p_i^{\alpha -1} p_i}{\alpha -1} = - \log \| p \|_{\alpha-1,p}
\end{align}
where for a positive function $q$ and measure $\mu$ on a discrete set, and $s \in (-\infty,0) \cup (0,\infty)$ we define 
\begin{align}
\|q\|_{s,\mu} = \left( \sum_i q_i^s \mu_i \right)^{\frac 1 s}.
\end{align}
Thus it follows from Jensen's inequality that $H_p(X)$ is non-increasing in $\alpha$.  The same argument in the continuous setting obviously holds. Thus, for $\alpha < \beta$ 
\begin{align}
    H_\beta(X) &\leq H_\alpha(X)    \label{eq: Discrete monotonicity of Renyi entropy}
        \\
    h_\beta(Y) &\leq h_\alpha(Y).    \label{eq: continuous monotonicity of Renyi entropy}
\end{align}
It has been known since \cite{BM11:it} (see also \cite{BN12}) that inequality \eqref{eq: continuous monotonicity of Renyi entropy} can be reversed up to an additive constant for log-concave random variables. 
\begin{theorem}[Bobkov-Madiman, \cite{BM11:it}] \label{thm:BM}
Let $Y$ be random vector in $\mathbb{R}^d$ with density $f:\mathbb{R}^d \to [0,\infty)$ such that $f((1-t)x+ty) \geq f^{1-t}(x) f^t(y)$ for all $t \in (0,1)$ and $x,y \in \mathbb{R}^d$. Then
\begin{align}
    h(Y) \leq h_\infty(Y) + d.
\end{align}
\end{theorem}
Recently (see \cite{FMW16,LFM2019}), a sharp comparison result has been also established between entropies of arbitrary orders.
\begin{theorem}[Fradelizi-Li-Madiman-Wang, \cite{FMW16,LFM2019}]\label{thm:FLMW}
Under the assumptions of Theorem \ref{thm:BM}, for $\alpha < \beta$, we have
\begin{align}
    h_\alpha(Y) - h_\beta(Y) \leq h_\alpha(Z) - h_\beta(Z),
\end{align}
where $Z$ is a random variable with density $e^{-\sum_{i=1}^d x_i}$ on $(0,\infty)^d$.
\end{theorem}

The key to this theorem is the following fact: for $f$ as in Theorem \ref{thm:BM}, the function $t \mapsto \log(t \int f^t)$ is concave on $(0,+\infty)$.

%Let us mention, though we will not pursue it further, that in \cite{BM11:it} these inequalities were equivalently stated in terms of a Gaussian comparison.  
%Bobkov and Madiman wrote that if a random vector $Y$ in $\mathbb{R}^d$ has a log-concave density $f$, let $Z$ in $\mathbb{R}^d$ be any normally distributed random vector with maximum density being the same as that of $Y$. Then
%\begin{align}
%h(Z) - \frac d 2 \leq h(Y) \leq h(Z)+ \frac d 2.
%\end{align}

Our pursuit in this paper will be the reversal of \eqref{eq: Discrete monotonicity of Renyi entropy} for the log-concave random variables on $\mathbb{Z}$.  That is random variables $X$ which are unimodal, and such that the sequence $p_n = \mathbb{P}(X=n)$ satisfies $p_i^2 \geq p_{i-1} p_{i+1}$. We say that such random variables are monotone (or that their distributions are monotone) if the sequence $(p_n)$ is monotone. 
We will prove the following main theorem.\\

\begin{theorem} \label{thm: infinity comparison}
For $\alpha \in (0,\infty)$, a log-concave random variable $X$, and $Z_p$ a geometric random variable with parameter $p$,
\begin{align}
    H_\alpha(X) - H_\infty(X) 
        &< 
            \lim_{p \to 0} H_\alpha(Z_p) - H_\infty(Z_p)
                \\
        &=
            \log \alpha^{\frac 1 {\alpha -1}} \label{eq: infinity renyi and alpha}.
\end{align}
\end{theorem}

This provides a discrete analog of Bobkov and Madiman's result, with a sharp constant provided by an extremizing sequence. Let us note, this is distinct from the continuous setting where for instance, an exponential distributions extremizes $h(X) - h_\infty(X)$ over log-concave distributions on $\mathbb{R}$.  It is consequence of the proof of Theorem \ref{thm: infinity comparison}, that $H_\alpha(X) - H_\infty(X) < \log \alpha^{\frac 1 {\alpha -1}}$, for $X$ log-concave, and thus that no such extremizer exists in the discrete setting.  

The proof methods in the discrete setting are distinct as well. In the continuous setting, there are analytic and convexity tools that can be leveraged, the Pr\'ekopa-Leindler inequality, integration by parts, perspective function convexity, etc. 

Though there has been significant recent interest in developing discrete versions of these continuous techniques (see \cite{Halikias,Kla-Leh,Gozlan-Rob-Sam-Te, Marsiglietti2020discretelocalization}), the authors are not aware of a version that could deliver Theorem \ref{thm: infinity comparison}.  Our proof method will show that any log-concave distribution majorizes a ``two-sided geometric distribution'' with the same $\infty$-R\'enyi entropy.  Using the Schur concavity of the R\'enyi entropy the problem is reduced to direct computation on these two-sided geometric distributions.  %The proof of Theorem \ref{thm: infinity comparison} finds a ``two sided'' geometric distribution majorizing a general log-concave distribution and then argues through direct computation.  

First progress on this result was made by the authors in \cite{MT20:isit} by a similar argument, which only applied however to monotone log-concave variables, and was leveraged by the concavity of Shannon entropy to deliver a sub-optimal reversal, $H(X) - H_\infty(X) \leq \log 2e$ for general log-concave variables. In contrast the results here, $H_\alpha(X) - H_\infty(X) \leq \log \alpha^{\frac 1 {\alpha - 1}}$ are sharp for all $\alpha$, and they apply to general log-concave variables, without the additional assumption of monotonicity.

As an application we investigate a conjecture of Madiman and Kontoyannis \cite{MK17}, which can be regarded as an entropic analog of the Rogers-Shephard inequality from Convex Geometry.  We consider general orders of the R\'enyi entropy, where the inequality can be considered a reversal, under the assumption of log-concavity, of the R\'enyi entropy power inequalities that have attracted recent attention (see \cite{BC14,  BC15:1, bobkov2017variants,  Jiange, li2019renyi, LMM19,  MMX17:2,MM18-IEEE, marsiglietti2018renyi, RS16, rioul2018renyi}).\\

Let us outline the paper. In the next section we gather basic definitions and properties regarding log-concave sequences and majorization. In Section \ref{sec: two sided geos} we prove the desired inequalities in the case that the distribution considered is two-sided geometric.  In Section \ref{sec: Proof} we show that every log-concave distribution majorizes a two-sided geometric distribution with the same min-entropy, and use the Schur concavity of the R\'enyi entropy to reduce the problem to the result proven in Section \ref{sec: two sided geos}.  In Section \ref{sec: Varentropy conjecture} we discuss a conjecture which would imply a sharp reversal of \eqref{eq: Discrete monotonicity of Renyi entropy} for all R\'enyi entropies (not just with comparison to min-entropy), but only for monotone log-concave distributions.  This conjecture amounts to the discrete analog of a result that does hold for all log-concave densities in $\mathbb{R}^d$, however as we show, in the discrete case, the result fails without the assumption of monotonicity.  In Section \ref{sec: RS}, we establish R\'enyi versions of the Rogers-Shephard inequality in both the continuous and discrete cases.

\paragraph{Acknowledgements.} 
We would like to thank Mokshay Madiman and Arnaud Marsiglietti for fruitful discussions throughout the development of the paper.
TT's research is partially supported by NSF grant DMS-1955175.

\section{Preliminaries} \label{sec: prelims}

We will use  the following notation for integer intervals.
For $a \leq b \in \mathbb{Z}$, $\llbracket a, b \rrbracket \coloneqq \{ x \in \mathbb{Z}: a \leq x \leq b \}$, $\rrbracket a, b \llbracket \coloneqq \{x \in \mathbb{Z}: a < x < b \}$, $\llbracket a, b \llbracket \coloneqq \{x \in \mathbb{Z}: a \leq x < b \}$ and so on.  We also let $\llbracket a, \infty \llbracket$ denote $\{ x \in \mathbb{Z}: a \leq x  \}$

\begin{definition}[Log-concavity]\label{def:  log-concavity}
    A sequence $f: \mathbb{Z} \to [0,\infty)$ is log-concave when it satisfies
    \begin{align} \label{eq: simple log-concave equation}
        f^2(n) \geq f(n-1) f(n+1)
    \end{align}
    for all $n \in \mathbb{Z}$ and $a \leq b \subseteq \{ f > 0\}$ implies $\llbracket a, b \rrbracket \subseteq \{ f > 0 \}$  (in other words, the support of $f$ is a contiguous interval of integers). A $\mathbb{Z}$-valued random variable $X$ is log-concave when the sequence $p_i = \mathbb{P}(X=i)$ is log-concave.  We will denote the space of all log-concave probability densities on $\mathbb{Z}$ by $\mathcal{L}(\mathbb{Z})$.
\end{definition}

There have already been a few interesting information theoretic results regarding log-concavity in discrete settings, see for example \cite{Joh07, JKM08:maxent}, but there is a vast mathematical literature, we mention only \cite{Sta89} and recall some relevant facts.
The class of log-concave sequences is closed under convolution, and thus the log-concave random variables are closed under independent summation.  The class is also closed under weak limits.  Important examples of log-concave distributions are the Bernoulli, Binomial, Geometric, Negative Binomial, and Poisson  distribution.

% {\blue
% \marginpar{Is this proposition ever used?}
% \begin{proposition}
% A sequence $f$ on $\mathbb{Z}$ is log-concave iff it satisfies
% \[
%     f(k + N) f(k+M) \geq f(k) f(k+M+N)
% \]
% for $k, N, M \in \mathbb{Z}$ and $NM \geq 0$ 
% \end{proposition}

% \begin{proof}
% In the if direction, take $k = n-1$, $N = M =1$,  \eqref{eq: simple log-concave equation}.  That its support is contiguous for $a < b$, satisfying $f(a) f(b) > 0$ take $k = a$  $N=1$ $m = b-a -1$, to see that $ f(a+1) f(b - 1) > 0$ as well and induct.  For the converse, note that when $f(k+M) f(k+M-1) \cdots f(k+1) f(k) >0$, \eqref{eq: simple log-concave equation} gives
% \[
%     \frac{f(k +1)}{f(k)} \geq \frac{f(k+2)}{f(k+1)} \geq \cdots \geq \frac{f(k+1+M)}{f(k+M)}
% \]
% Multiplying the inequality for $N \geq 0$, when $f(k) f(k+1) \cdots f(k+M+N-1) > 0$
% \begin{align}
%      \frac{ f(k+N)}{f(k)} 
%         &= 
%             \prod_{l=0}^{N-1} \frac{f(k +l+1)}{f(k+l)} 
%                 \\
%         &\geq 
%             \prod_{l=0}^{N-1} \frac{f(k +M+l+1)}{f(k+M+l)} 
%                 \\
%         &= \frac{f(k+M+N)}{f(k+M)}.
% \end{align}
% \end{proof}

% }

Let $\ell_1(\mathbb{Z})$ denote the functions $f: \mathbb{Z} \to \mathbb{R}$ such that $\sum_{i \in \mathbb{Z}} |f(i)| < \infty$.

\begin{definition}[Decreasing rearrangment] \label{def: decreasing rearrangement}
For a function $f: \mathbb{Z} \to [0,\infty)$ in $\ell_1(\mathbb{Z})$ denote the sequence $f_i^\downarrow$ to be the decreasing rearrangment of $f$.  Explicitly, $f^\downarrow$ satisfies $f_i^\downarrow \geq f_{i+1}^\downarrow$ for all $i$ and $f_i^\downarrow = f(\tau(i))$ for a bijection $\tau: \mathbb{N} \to \mathbb{Z}$.  
\end{definition}

\begin{definition}[Majorization] \label{def: majorization}
    For $\ell_1(\mathbb{Z})$ functions $f,g: \mathbb{Z} \to [0,\infty)$, we say $f$ majorizes $g$ and write $f \succ g$ when 
    \begin{align}
        \sum_{i=1}^k f_i^\downarrow \geq \sum_{i=1}^k g_i^\downarrow,
    \end{align}
    and equality holds in the limit with $k \to \infty$.
\end{definition}

We let $\mathcal{P}(\mathbb{Z})$ denote the space of probability densities on $\mathbb{Z}$.
\begin{definition}[Schur Convexity] \label{def: Schur convexity}
    A function $\Phi : \mathcal{P}(\mathbb{Z}) \to \mathbb{R} \cup \{ \infty\}$ is Schur-convex when $ f \succ g$ implies
    \begin{align}
        \Phi(f) \geq \Phi(g).  
    \end{align}
    $\Phi$ is Schur-concave when $-\Phi$ is Schur-convex.
\end{definition}

\begin{definition}[Two-sided geometric distribution] \label{def: two-sided geometric}
A density function $\varphi$ on $\mathbb{Z}$ is a two-sided geometric distribution when there exists $p, q \in [0,1)$ and $m \in \mathbb{Z}$ such that its density function $\varphi$ can be expressed as
\begin{align}
    \varphi(n) = \frac{(1-p)(1-q)}{1-pq} f(n-m).
\end{align}
with
\begin{align}
    f(n) = \begin{cases}
                p^{n} & \hbox{ for } n\geq 0 \\
                q^{-n} & \hbox{ for } n \leq 0.
    \end{cases}
\end{align}
with the convention that $0^0 = 1$.  We will denote the set of all such densities by $\mathcal{E}(\mathbb{Z})$.
\end{definition}
For $q = 0$ and $m = 0$, $\varphi$ is the usual geometric distribution with parameter $1-p$, and when $p=q=0$ we have a point mass at $m$.  In any case, observe that all two-sided geometric distributions are log-concave, since their support is by definition contiguous, $\varphi^2(n) = \varphi(n+1) \varphi(n-1)$ for $n \neq m$ and $\varphi^2(m) = \max_n \varphi^2(n) \geq \varphi(m-1) \varphi(m+1)$.

\iffalse
\begin{align}
    \sum_{n \in \mathbb{Z}} f(n) 
        &=         \sum_{n=0}^\infty p^n + \sum_{n=1}^\infty q^n
            \\
        &=
            \frac 1 {1-p} + \frac{q}{1-q}
                \\
        &=
            \frac 1 {1-p} + \frac 1 {1-q} -1
                \\
        &=
            \frac{1-pq}{(1-p)(1-q)}
\end{align}
Thus we can define a density,

Thus $\|\varphi\|_\infty = \frac{(1-p)(1-q)}{1-pq}.$

\begin{definition}[Majorization]
 and write $a \succ b$ when,
\begin{align}
    \sum_{i=1}^k a_i^\downarrow \geq \sum_{i=1}^k b_i^\downarrow 
\end{align}
holds for all $k$. 
\end{definition}
\fi

\section{Two-sided geometric distributions} \label{sec: two sided geos}
In this section we will derive bounds on the R\'enyi entropy on the two-sided geometric distribution.

\begin{lemma} \label{lem: two sided geo renyi bound}
    If $\varphi$ is a density function on $\Z$ with a two sided geometric distribution with parameters $p$ and $q$, then for $\alpha \in (0,\infty)$, we have
    \begin{align}
        H_\alpha(\varphi) - H_\infty(\varphi) < \frac{\log \alpha}{\alpha - 1}.
    \end{align}
    with $\frac{\log \alpha }{\alpha - 1} \big|_{\alpha = 1} \coloneqq \log e$.
\end{lemma}

\begin{proof}
First we proceed with the case $\alpha \neq 1$.
By direct computation,
    \begin{align}
    H_\alpha(\varphi) - H_\infty(\varphi)
        &=
            (1-\alpha)^{-1}\log \left( \sum_n \varphi^\alpha (n) \right) + \log \| \varphi \|_\infty
                \\
        &=
            (1-\alpha)^{-1} \left( \log \| \varphi \|_\infty^\alpha + \log \left( \frac 1 {1-p^\alpha} + \frac 1 {1-q^\alpha} - 1 \right)  \right) + \log \| \varphi \|_\infty
                \\
        &=
            \frac{ \log \| \varphi \|_\infty + \log \left( \frac 1 {1-p^\alpha} + \frac 1 {1-q^\alpha} - 1 \right) }{1-\alpha}
                \\
        &=
            \frac{\log \left( \frac{\frac 1 {1-p^\alpha} + \frac 1 {1-q^\alpha} - 1}{\frac 1 {1-p} + \frac 1 {1-q} - 1} \right)}{ 1- \alpha}.
\end{align}
Thus, it suffices to prove for $\alpha > 1$
    \begin{align}
       \frac 1 {1-p} + \frac 1 {1-q} - 1  <  \alpha \left(\frac 1 {1-p^\alpha} + \frac 1 {1-q^\alpha} - 1 \right)
    \end{align}
    and 
    \begin{align}
       \frac 1 {1-p} + \frac 1 {1-q} - 1  >  \alpha \left(\frac 1 {1-p^\alpha} + \frac 1 {1-q^\alpha} - 1 \right)
    \end{align}
    when $\alpha \in (0,1)$.  Note that we have equality when $\alpha = 1$, so it suffices to show the function,
    \begin{align}
        F(\alpha) = \alpha \left(\frac 1 {1-p^\alpha} + \frac 1 {1-q^\alpha} - 1 \right)
    \end{align}
    is strictly increasing.  Computing directly,
    \begin{align}
        F'(\alpha) = \frac{p^\alpha \log p^\alpha}{(1-p^\alpha)^2} + \frac{q^\alpha \log q^\alpha}{(1-q^\alpha)^2} + \frac{ 1 }{1 - p^\alpha} + \frac 1 {1 - q^\alpha} - 1.
    \end{align}
Writing $x = p^\alpha$ and $y = q^\alpha$ it is enough to prove
\begin{equation}\label{eq:ineq-xy}
    \frac{x \log x}{(1-x)^2} + \frac{y \log y}{(1-y)^2} + \frac{ 1 }{1 - x} + \frac 1 {1 - y} > 1
\end{equation}
for $x,y \in (0,1)$. To this end, we consider $f(x) = \frac{x\log x}{(1-x)^2} + \frac{1}{1-x}$, $x \in (0,1)$ and argue that $f(x) > \frac{1}{2}$ for $x \in (0,1)$. Note that 
$
(1-x)^3f'(x) = 2(1-x)+(1+x)\log x
$
and the right hand side, call it $g(x)$, is an increasing function of $x$ on $(0,1)$ because
$
g'(x) = \log x + \frac{1}{x}-1 = -\log\frac{1}{x} + \frac{1}{x} - 1 > 0.
$
Thus $g(x) < g(1) = 0$ for $x \in (0,1)$ which shows that $f'(x) < 0$ for $x \in (0,1)$. Thus $f$ is strictly decreasing and we get $f(x) > f(1-) = \frac{1}{2}$ for $x \in (0,1)$.

When $\alpha = 1$, a direct computation gives 
\[
H(\phi) - H_\infty(\phi) = -\frac{\frac{p\log p}{(1-p)^2} + \frac{q\log q}{(1-q)^2}}{\frac{1}{1-p}+ \frac{1}{1-q} - 1}
\]
which is strictly less than $1$ by \eqref{eq:ineq-xy}.
\end{proof}

\iffalse
\begin{align}
    \Psi(p, q) =  -  \frac{(1-p)(1-q)}{1-pq}\left(  \frac{p \log p}{(1-p)^2} +  \frac{q \log q} {(1-q)^2} \right)
\end{align}
for $(p,q) \in [0,1]^2$.
\fi

\section{Proof of Theorem \ref{thm: infinity comparison}} \label{sec: Proof}

\begin{lemma}  \label{lem: majorization}
For a non-negative monotone log-concave $\ell_1(\Z)$ function $f$, supported on $\rrbracket n, k \llbracket$, there exists a unique log-affine positive function $g$ with infinite set of support containing $\rrbracket n, k \llbracket$, the same mode, maximum value, and $\ell_1(\Z)$ norm such that $ f \succ g$.
\end{lemma}

\begin{proof}
Without loss of generality, let us assume that $f$ is non-increasing and supported on $\llbracket 0, k \llbracket$.  Define for $q \in [0,1)$ a new function $g_q$ to be log-affine supported on $\llbracket 0 , \infty \rrbracket$ and such that $g_q(1) = f(1)$ and $g(j + 1) = q g(j)$.  Note that $\sum_j g_0(j) \leq \sum_j f(j) < \lim_{q \to 1} \sum_j g_q(j) = \infty$ and the function $q \mapsto \sum_j g_q(j)$ is strictly increasing in $q$.  Thus, by the intermediate value theorem, there exists a unique $q^* \in [0,1]$ such that $\sum_j g_{q^*}(j)= \sum_j f(j)$.  Take $g = g_{q^*}$, and observe that $g$ is log-affine and satisfies $\|g\|_\infty = \|f\|_\infty$ by construction.  Since $g$ is log-affine and $f$ is log-concave, for the function $f-g$ there exists and $l \in \llbracket 1, n \rrbracket$ such that $f(k)-g(k) \geq 0$ on $\llbracket 1 , l \rrbracket$ and $f(k)-g(k) \leq 0$.
\end{proof}

Using the notation $\mathcal{L}_m(\mathbb{Z}) = \{ f \in \mathcal{L}(\mathbb{Z}): H_\infty(f) = m\}$ and $\mathcal{E}_m(\mathbb{Z}) = \{ f \in \mathcal{E}(\mathbb{Z}): H_\infty(f) = m\}$, we will prove the following general result, that roughly states minimizers of Schur-convex functions among log-concave densities satisfying a constraint on their maximum value, are two-sided geometric distributions.

\begin{theorem} \label{thm: Schur extremizers}
    For a Schur convex function, $\Phi: \mathcal{P}(\mathbb{Z}) \to \mathbb{R}$, and $f \in \mathcal{L}_m(\mathbb{Z})$ there exists $\varphi \in \mathcal{E}_m(\mathbb{Z})$ such that 
    \begin{align}
        \Phi(f) \geq \Phi(\varphi).
    \end{align}
\end{theorem}

\begin{proof}
    We will prove the result by showing that given a log-concave density $f$ there exists a two-sided geometric distribution $\varphi$ with the same maximum such that $f \succ \varphi$.  Without loss of generality, suppose that $f(0) = \|f\|_\infty$ and $t \coloneqq \sum_{j=0}^\infty f(j) \geq \sum_{j= -\infty}^{0} f(j)$, note $t \geq \frac 1 2$.  
%By the monotonicity of R\'enyi entropy, it suffices to consider $\alpha = 1$ and $\beta = \infty$. 
Define $f_+(j) = \mathbbm{1}_{[0,\infty)}(j) f(j)$ and $f_-(j) = \mathbbm{1}_{(-\infty,0]}(j) f(j)$, define $a_+$ to be the log-affine function of the same maximum and majorized by $f_+$, while $a_-$ to be the log-affine function of the same maximum and majorized by $f_-$, as supplied by Lemma \ref{lem: majorization}. Note that $a_+(j) =  f(0) p^j$ for $j \geq 0$ and $p \in [0,1)$ while $a_-(k) =  f(0) q^{-j}$ for $j \leq 0$ for $q \in [0,1)$.  Note that by our assumptions $\sum_{j=0}^\infty a_+(j) = \sum_{j=0}^\infty f(j) \geq \sum_{j= -\infty}^{0} f(j) = \sum_{j=-\infty}^0 a_-(j)$ implies $p \geq q$.  Define
\begin{align}
    \varphi(j) = \begin{cases}
                a_+(j),& j > 0 
            \\
            f(0), &j =0
                \\
            a_-(j), & j < 0.
            \end{cases}
\end{align}
By construction $\sum_i \varphi_i = \sum_i f(i)$. In particular, $\varphi$ is a density function on $\Z$. Observe that $\varphi$ is the two-sided geometric distribution majorized by $f$ as desired, since for every $n$ 
\begin{align}
    \sum_{i=1}^{n} \varphi^\downarrow_i 
        &=
            f(0) + \sum_{i=1}^{j} a_+(i) + \sum_{i=1}^{n-j-1} a_-(i)
                \\
        &\leq
            f(0) + \sum_{i=1}^{j} f_+(i) + \sum_{i=1}^{n-j-1} f_-(i)
                \\
        &\leq 
            \sum_{i=1}^n f^\downarrow_{i}
\end{align}
Thus $ \varphi  \prec f$ and by Schur-convexity
 $\Phi(\varphi) \leq \Phi(f)$, and thus the proof is complete.
\end{proof}

\begin{proof}[Proof of Theorem \ref{thm: infinity comparison}]
As is well known, see \cite{Marshall-Olkin, ho2015convexity} the R\'enyi entropy is Schur-concave.  Given $f \in \mathcal{L}_m(\mathbb{Z})$, by Theorem \ref{thm: Schur extremizers} there exists $\varphi \in \mathcal{E}_m(\mathbb{Z})$ such that
\begin{align}
    H_\alpha(f) 
        &\leq 
            H_\alpha(\varphi)
                \\
        &<
            H_\infty(\varphi) + \frac{\log \alpha}{\alpha - 1} \label{eq: inequa here}
                \\
        &=
            H_\infty(f) + \frac{\log \alpha}{\alpha - 1},
\end{align}
where \eqref{eq: inequa here} follows from Lemma \ref{lem: two sided geo renyi bound}.  This gives the strict inequality of the theorem. To show that it is attained in the limit of geometric distribution is an easy and direct computation.
\end{proof}

\section{Varentropy and logconcavity of auxiliary function}\label{sec: Varentropy conjecture}
In analogy to the continuous case, we pose the following conjecture.

\begin{conj}\label{conj:sum}
Let $(x_n)_{n=1}^N$ be a finite monotone log-concave sequence. Then the function
\[
F(t) = \log\left[t\sum_{n=1}^N x_n^t\right]
\]
is concave on $(0,+\infty)$. 
\end{conj}

By looking at the second derivative, this conjecture is equivalent to the statement that for every finite monotone log-concave sequence $(y_n)_{n=1}^N$ ($y_n = x_n^t$), we have
\[
\frac{\left(\sum y_n\log^2y_n\right)\left(\sum y_n\right)-\left(\sum y_n\log y_n\right)^2}{\left(\sum y_n\right)^2} \leq 1
\]
(the left-hand side can be seen as the varentropy -- see \cite{FMW16}).

Note that Conjecture \ref{conj:sum} implies the desired sharp comparison between R\'enyi entropies of arbitrary two orders for monotone log-concave random variables. Indeed, suppose $X$ is such a random variable. By an approximation argument, we can assume that the support of $X$ is finite, say it is $\{1,\cdots,N\}$ and $p_n = \mathbb{P}(X = n) > 0$, $n=1,\cdots,N$, is log-concave. We have
\begin{equation}\label{eq:FH}
H_\alpha(X) = \frac{1}{1-\alpha}\log\left[\sum_{n=1}^N p_n^\alpha\right] = \log\left(\alpha^{\frac{1}{\alpha-1}}\right) + F(\alpha),
\end{equation}
where $F(\alpha)$ is the function from Conjecture \ref{conj:sum} for the sequence $(p_n)$. Suppose $1 < \alpha < \beta$ and write $\alpha = 1-\lambda + \lambda \beta$ with $\lambda = \frac{\alpha-1}{\beta-1} \in (0,1)$. If $F$ was concave, then we would have
\[
F(\alpha) = F((1-\lambda)\cdot 1 + \lambda \beta) \geq (1-\lambda)F(1) + \lambda F(\beta) = \lambda F(\beta) 
\]
which, by \eqref{eq:FH}, becomes
\begin{equation}\label{eq:H-opt}
H_\alpha(X) - H_\beta(X) \leq \log\left(\frac{\alpha^{\frac{1}{\alpha-1}}}{\beta^{\frac{1}{\beta-1}}}\right).
\end{equation}
Proceeding in a similar way, this would also follow for $\alpha < \beta < 1$ as well as $\alpha < 1 < \beta$.

As opposed to the continuous case, Conjecture \ref{conj:sum} cannot hold without the monotonicity assumption: for example, for the sequence $p = (\frac{1}{4},\frac{1}{2},1,\frac{1}{2},\frac{1}{4})$ the function $F$ is not concave because we have $F''(3) > 0.0009$. This is a  consequence of the fact that for a symmetric geometric random variable $X$ with $\mathbb{P}(X = k) = p^{|k|}$, $k \in \Z$, and $\alpha < \beta$, the opposite inequality to \eqref{eq:H-opt}  holds.

In view of what is true in the continuous case \cite{cohn1969some} we propose the following strengthening of Conjecture \ref{conj:sum}.

\begin{conj}\label{conj:sum-strong}
Let $(y_n)_{n=1}^N$ be a finite positive monotone and concave sequence, that is $y_n \geq \frac{y_{n-1}+y_{n+1}}{2}$, $1 < n < N$. Then for every $\gamma > 0$, the function
\[
K(t) = (t+\gamma)\sum_{n=1}^N y_n^{t/\gamma}
\]
is log-concave, that is $\log K(t)$ is concave on $(-\gamma,+\infty)$. 
\end{conj}
The approach from \cite{cohn1969some} also suggests that the following stronger inequality be true: for every complex number $z = u+iv$ with $u > -\gamma$, we have $|K(z)| \geq K(u)$. It  turns out that it implies Conjecture \ref{conj:sum-strong} (see \cite{cohn1969some} for a detailed argument).

As a final remark, we also point out that Conjecture \ref{conj:sum} has an  application to Khinchine-type inequalities: it would furnish a large family of examples of the so-called ultra sub-Gaussian random variables, see Remark 13 in \cite{havrilla2019sharp}. Conjecture \ref{conj:sum} was verified therein for sequences of length $3$.

\section{Application to R\'enyi Entropic Rogers-Shephard Inequality}\label{sec: RS}
It is a classical theorem of Convex Geometry called the Rogers-Shephard Inequality \cite{RS57} that for a convex body $K \subseteq \mathbb{R}^d$ (a compact convex set with non-empty interior), we have
\begin{align} \label{eq: Rogers-Shephard}
    Vol(K-K) \leq \binom{2d}{d} Vol(K),
\end{align}
with equality if and only if $K$ is the $d$-dimensional simplex.  
This can be easily rephrased as a R\'enyi entropic inequality if we recall the usual definition of the R\'enyi entropy for continuous variables.
The Rogers-Shephard inequality thus says that for independent $X$ and $Y$ with a common density function $f$ on $\mathbb{R}^d$ supported on a convex body in $\R^d$,
\begin{align}
    h_0(X-Y) \leq h_0(X) + \log {2d \choose d }.
\end{align}
Note that by Stirling's formula $\log {2d \choose d} \approx d \log 4$ and the best possible dimension independent $c$ such that $h_0(X-Y) \leq h_0(X) + d \log c$ holds is $c = 4$.

An entropic analog of Rogers-Shephard has been pursued in \cite{madiman2016entropy}, where the following is conjectured.
\begin{conj}[Madiman-Kontoyannis \cite{madiman2016entropy}]\label{conj:M-K}
For log-concave $X$ and $Y$ iid $\mathbb{R}^d$ valued random variables
\begin{align}\label{eq:M-K}
    h(X-Y) \leq h(X) + d \log {2}
\end{align}
with the equality case given by the $d$-dimensional exponential product distribution.
\end{conj}

In the same article Madiman and Kontoyannis prove that under the same hypotheses 
$
    h(X-Y) \leq h(X) + d \log {4},
$
holds.  Moreover using alternate methods in \cite{BM13:goetze} a bound of 
$
    h(X-Y) \leq h(X) + d \log e.
$
was obtained.
In the continuous case we present the following generalization to the R\'enyi entropy.
\begin{theorem} \label{thm: Renyi entropic Rogers-Shephard}
Let $X$ and $Y$ be iid log-concave random vectors in $\R^d$. If $\alpha \in [2,\infty]$, then
\begin{align}\label{eq:RS}
        h_\alpha(X-Y) \leq  h_\alpha(X) + d \log 2,
    \end{align}
    with equality when $X$ has exponential distribution $\mathbbm{1}_{(0,\infty)^d}(x) e^{- \sum_i x_i}$.
    If $\alpha  \in [0,2]$, then
    \begin{align}
        h_\alpha(X-Y)  \leq h_\alpha(X) + d \log \alpha^{\frac 1 {\alpha-1}}.
    \end{align}
\end{theorem}
Note that when $\alpha = 1$ we recover \cite{BM13:goetze}, the sharpest known bound in the Shannon case $h(X-Y) \leq h(X) + d \log e$.
We will need the following corollary of Theorem \ref{thm:FLMW}.

\begin{cor} [Fradelizi-Madiman-Wang \cite{FMW16}] \label{cor: FMW}
    For $0 \leq \alpha \leq \beta$, and $X$ a log-concave vector in~$\mathbb{R}^d$, we have
    \begin{align}
        h_\beta(X) \leq h_\alpha(X) \leq h_\beta(X) + d \log \frac{c(\alpha)}{ c(\beta) }
    \end{align}
    where
    \begin{align}
        c(\alpha) = \alpha^{\frac 1 {\alpha-1}} ,
    \end{align}
    with $c(\infty) \coloneqq 1, c(1) \coloneqq e$, and $c(0) \coloneqq \infty$.
    % so that 
    % \begin{align}
    %     c(\alpha, \beta) c(\beta, \gamma) = c(\alpha, \gamma)
    % \end{align}
    % for $0 \leq \alpha \leq \beta \leq \gamma \leq \infty$.
\end{cor}

\begin{proof}
   If $Z$ is a random variable with density $e^{-x}\mathbbm{1}_{(0,\infty)}$, then $h_\alpha(Z) = \frac{\log \int_0^\infty e^{-\alpha x} dx}{1-\alpha} = \log \alpha^{\frac 1 {\alpha - 1}}$.  
 Let $Z_d$ be a random vector in $\R^d$ with density $e^{- \sum_i x_i}\mathbbm{1}_{(0,\infty)^d}$. By Theorem \ref{thm:FLMW},
   \begin{align}
       h_\alpha(X) &\leq h_\beta(X) + h_\alpha(Z_d) - h_\beta(Z_d)
   \end{align}
and $h_\alpha(Z_d) - h_\beta(Z_d)= 
                d (h_\alpha(Z) - h_\beta(Z))
            =
                d \log \frac{c(\alpha)}{c(\beta)}.$
\end{proof}

\begin{proof}[Proof of Theorem \ref{thm: Renyi entropic Rogers-Shephard}]
The result hinges on the following equality,
\begin{align}\label{eq:oo-2}
    h_\infty(X-Y) = h_2(X).
\end{align}
Explanation: letting $f$ denote the shared density of $X$ and $Y$, the density function $f_{X-Y}$ of $X-Y$ is given by
\begin{align}
    f_{X-Y}(z) = \int_{\mathbb{R}^d} f(z - x) f(-x) dx  = \int_{\mathbb{R}^d} f(z+x) f(x) dx,
\end{align}
which is log-concave and even, thus
$    \|f_{X-Y}\|_\infty = f_{X-Y}(0) = 
            \int_{\mathbb{R}^d} f^2(x) dx.
$

For $\alpha \geq 2$,
\begin{align}
    h_\alpha(X - Y) &\leq h_\infty(X-Y) + d \log \frac{c(\alpha)}{c(\infty)}
        \\
            &=
                h_2(X) +d \log \frac{c(\alpha)}{c(\infty)}
                    \\
            &\leq \label{eq:RS-proof-1}
                h_\alpha(X) + d\log \frac{c(2)}{c(\alpha)} + d\log \frac{c(\alpha)}{c(\infty)}
                    \\
            &= \label{eq:RS-proof-2}
                h_\alpha(X) + d\log \frac{c(2)}{c(\infty)}.
\end{align}
Since $\frac{c(2)}{c(\infty)} = 2$, this completes the first proof.  Moreover, this result is sharp with equality for $\mathbbm{1}_{(0,\infty)^d} e^{- \sum_i x_i}$, which can be verified by checking the $d=1$ case and then tensorizing.

 When $\alpha < 2$,
 \begin{align}
    h_\alpha(X -Y) &\leq h_\infty(X-Y) + d \log \frac{c(\alpha)}{c(\infty)}
        \\
        &= h_2(X) + d \log \alpha^{\frac 1 {\alpha-1}}
            \\
        &\leq
            h_\alpha(X) + d \log \alpha^{\frac 1 {\alpha-1}}.
\end{align}
\end{proof}

\begin{remark}
Identity \eqref{eq:oo-2} was also crucial in \cite{Jiange} in Li's proof of the conjectural entropic Buseman theorem (see \cite{BNT15,MMX17:1}) in the case of R\'enyi entropy of order $2$. In fact, in the symmetric case \eqref{eq:RS} for $\alpha = 2$ follows from Li's result \cite{Jiange}. When $d=1$ and $X$ is symmetric, Conjecture \ref{conj:M-K} is a special case of \cite[Conjecture  1]{BNT15}.  Alternatively, under the symmetry assumption $X-Y$ has the same distribution as $X+Y$, and the result follows from the additive case \cite{CZ94,Yu08:2, MMX17:1} for $\alpha \leq 1$ even when $X$ and $Y$ are dependent.  Under stronger concavity assumptions, there are analogous results for $\alpha \geq 1$, see \cite{XMM16:isit,  LM18:isit, LMM19}.
\end{remark}

From Theorem \ref{thm: infinity comparison}, repeating the above proof mutatis mutandis, we can establish discrete analogs of the R\'enyi entropic Rogers-Shephard.

%{\blue \marginpar{case $\alpha=0$ causes $\infty - \infty$-like trouble}}
\begin{theorem} \label{thm: discrete Renyi entropic Rogers-Spephard}
    For $X$ and $Y$ iid log-concave variables on $\mathbb{Z}$,
    \begin{align}
        H_\alpha(X-Y) - H_\alpha(X) 
            < \log c(\alpha),
    \end{align}
    with 
    \begin{align}
        c(\alpha) = \begin{cases} 2 \alpha^{\frac 1 {\alpha-1}}, &\mbox{if } \alpha  \in (2,\infty], \\
 \alpha^{\frac 1 {\alpha - 1}}, & \mbox{if } \alpha \in (0,2].
 %\\
% 2 & \mbox{if } \alpha = 0.
\end{cases} 
    \end{align}
\end{theorem}
The limiting cases $1$ and $\infty$ are understood by the same conventions as Corollary \ref{cor: FMW}, $1^{\frac{1}{1-1}} \coloneqq e$ and $\infty^{\frac 1 {\infty - 1}} \coloneqq 1$.  In the case of the Shannon entropy this gives
    \begin{align}
        H(X-Y) - H(X) < \log e.
    \end{align} When $\alpha = 0$, $H_0(X-Y) \leq H_0(X) + \log 2$ holds, and is strict whenever the support of $X$ is finite.  The inequality is sharp as can be seen by taking a log-concave distribution supported on $\{0,1,\dots, n\}$ for $n$ large.

\begin{proof}
The only modification of the proof of Theorem \ref{thm: Renyi entropic Rogers-Shephard} is in \eqref{eq:RS-proof-1} and \eqref{eq:RS-proof-2}, where we use $H_2(X) < H_\infty(X) + \log \alpha^{\frac{1}{1-\alpha}}$ and $H_\infty(X) \leq H_2(X)$. 
\end{proof}

    Observe that when $X_p$ and $Y_p$ are iid geometric with parameter $p$, and  $\alpha > 0$, then \begin{align}
        \lim_{p \to 0} H_\alpha(X_p-Y_p) - H_\alpha(X_p) = \lim_{p\to 0} \frac{1}{1-\alpha}\log\frac{1+(1-p)^\alpha}{(2-p)^{\alpha}} = \log 2.
    \end{align}  
%    Further, when $\alpha = 0$, and $X_n$ and $Y_n$ are uniform on $\{1,2, \dots, n\}$ then $\lim_n h_\alpha(X_n-Y_n) - h_\alpha(X_n) = \log 2$.  
Thus $c(\alpha)$ cannot be improved beyond $2$ for any $\alpha$, and hence Theorem \ref{thm: discrete Renyi entropic Rogers-Spephard} is sharp for $\alpha \in \{2,\infty\}$.

\bibliographystyle{plain}
\bibliography{bibibi}
\end{document}